\documentclass[english]{amsart}

\usepackage{amsmath,amssymb,enumerate}
\usepackage{amscd,amsxtra,calc}

\usepackage[T1]{fontenc}
\usepackage[all]{xy}

\usepackage{babel}
\usepackage{amstext}
\usepackage{amsmath}
\usepackage{amsfonts}
\usepackage{latexsym}
\usepackage{ifthen}
\usepackage{verbatim}
\usepackage{enumerate}
\usepackage[stable]{footmisc}
\usepackage{xypic}
\xyoption{all}
\newcommand{\id}{{\rm id}}

\newtheorem{lemma1}{}[section]

\newenvironment{lemma}{\begin{lemma1}{\bf Lemma.}}{\end{lemma1}}
\newenvironment{example}{\begin{lemma1}{\bf Example.}\rm}{\end{lemma1}}

\newenvironment{theorem}{\begin{lemma1}{\bf Theorem.}}{\end{lemma1}}
\newenvironment{proposition}{\begin{lemma1}{\bf Proposition.}}{\end{lemma1}}
\newenvironment{corollary}{\begin{lemma1}{\bf Corollary.}}{\end{lemma1}}
\newenvironment{remark}{\begin{lemma1}{\bf Remark.}\rm}{\end{lemma1}}
\newenvironment{definition}{\begin{lemma1}{\bf Definition.}}{\end{lemma1}}

\newenvironment{conjecture}{\begin {lemma1}{\bf Conjecture.}}{\end{lemma1}}

\newenvironment{remark*}{{\bf Remark.}}{}
\newenvironment{example*}{{\bf Example.}}{}

\newcommand{\R}{\ensuremath{\mathbb{R}}}
\newcommand{\Q}{\ensuremath{\mathbb{Q}}}
\newcommand{\Z}{\ensuremath{\mathbb{Z}}}
\newcommand{\C}{\ensuremath{\mathbb{C}}}
\newcommand{\N}{\ensuremath{\mathbb{N}}}
\newcommand{\PP}{\ensuremath{\mathbb{P}}}

\newcommand{\holom}[3]{\ensuremath{#1:#2  \rightarrow #3}}
\newcommand{\fibre}[2]{\ensuremath{#1^{-1} (#2)}}

\makeatletter
\ifnum\@ptsize=0  \fi \ifnum\@ptsize=2  \fi 

\newcommand\sF{{\mathcal F}}

\newcommand\sO{{\mathcal O}}

\newcommand{\chow}[1]{\ensuremath{\mathcal{C}(#1)}}

\newcommand{\barS}{\ensuremath{\overline{S}}}

\newcommand{\blS}{\ensuremath{\widetilde{S \times S}}}

\renewcommand{\c}[0]{{\mathbb C}}  

\renewcommand{\o}[0]{{\mathcal O}} 
\newcommand{\z}[0]{{\mathbb Z}}

\renewcommand{\r}[0]{{\mathbb R}}

\newcommand{\pic}[0]{\operatorname{Pic}}

\newcommand{\supp}[0]{\operatorname{Supp}}

\title{Twisted cotangent bundles of Hyperk\"ahler manifolds}
\date{June 26, 2019}

\author {Fabrizio Anella}
\author {Andreas H\"oring}

\address{Fabrizio Anella\\ Dipartimento di Matematica e Fisica\\ Universit\`{a} Roma Tre\\ Italy}
\email{fabrizio.anella2@uniroma3.it}
\address{Andreas H\"oring, Universit\'e C\^ote d'Azur, CNRS, LJAD, France}
\email{Andreas.Hoering@univ-cotedazur.fr}

\begin{document}

\begin{abstract} 
Let $X$ be a Hyperk\"ahler manifold, and let $H$ be an ample divisor on $X$. We give a lower bound
in terms of the Beauville-Bogomolov form $q(H)$ for the twisted cotangent bundle $\Omega_X \otimes H$ to be 
pseudoeffective. If $X$ is deformation equivalent to the Hilbert scheme of a K3 surface the lower bound can be written down explicitly and we study its optimality. 
\end{abstract}

\maketitle

\section{Introduction}

\subsection{Motivation and main result}

Let $X$ be a compact K\"ahler manifold, and let $\Omega_X$ be the cotangent bundle of $X$.
If the canonical bundle $K_X = \det \Omega_X$ is positive (e.g. pseudoeffective or nef) we can use stability theory to describe the positivity of $\Omega_X$. The most famous result in this direction is Miyaoka's theorem \cite{Miy87}
which says that for a projective manifold that is not uniruled, the restriction $\Omega_X|_C$ to a general complete intersection curve $C$ of sufficiently ample general divisors is nef. However this result only captures a part of the picture:
denote by $\zeta \rightarrow \PP(\Omega_X)$ the tautological class on the projectivised
cotangent bundle $\holom{\pi}{\PP(\Omega_X)}{X}$. If $X$ is Calabi-Yau or a projective Hyperk\"ahler manifold 
the tautological class $\zeta$ is not pseudoeffective \cite[Thm.1.6]{HP19}. In particular $X$ is covered by curves $C$ 
such that $\Omega_X|_C$ is not nef.

Our goal is to measure this defect of positivity by considering polarised manifolds $(X, H)$. 
This has been accomplished for infinitely many families of projective K3 surfaces in a beautiful paper
of Gounelas and Ottem:

\begin{theorem}\label{gounelas} \cite[Thm.B]{GO18}
Let $(X, H)$ be a primitively polarised K3 surface of degree $d$ and Picard number one. 
Denote by $\holom{\pi}{\PP(\Omega_X)}{X}$ the projectivisation of the cotangent bundle, and by $\zeta \rightarrow \PP(\Omega_X)$ the tautological class.

Suppose that $\frac{d}{2}$ is a square and the Pell equation
$$
 x^2-2 dy^2=5 
$$
has no integer solution. Then $\zeta+\frac{2}{\sqrt{\frac{d}{2}}}\pi^*H$ is pseudoeffective and
$\zeta+(\frac{2}{\sqrt{\frac{d}{2}}}-\varepsilon)\pi^*H$ is not pseudoeffective for any $\varepsilon>0$.
\end{theorem}

In the situation above one has $(\frac{2}{\sqrt{\frac{d}{2}}}H)^2 = 8$, so we see that, under these numerical conditions, the class $\zeta+\pi^* H$
is pseudoeffective for an ample $\R$-divisor class $H$ of degree at least eight.
In view of this observation we make the following

\begin{conjecture} \label{conjecturegeneral}
Fix an even natural number $2n$. Then there exists only 
finitely many deformation families of polarised Hyperk\"ahler manifolds $(X, H)$ such that $\dim X=2n$ and
$H$ is ample Cartier divisor on $X$ such that $\zeta + \pi^* H$ is not pseudoeffective.
\end{conjecture} 

This conjecture should be seen as an analogue of the situation for uniruled manifolds: in this case $\Omega_X$
is not even generically nef in the sense of Miyaoka, but $\Omega_X \otimes H$ is
generically nef unless $X$ is very special (\cite[Thm.1.1]{Hoe14}, see \cite[Cor.1.3]{AD17} for a stronger version). 

In this paper we give a sufficient condition for the pseudoeffectivity of twisted cotangent bundles
for Hyperk\"ahler manifolds. Since deformations to non-projective Hyperk\"ahler manifolds
are crucial for the proof we state the result in the analytic setting:

\begin{theorem}\label{theoremmain}
Let $X$ be a (not necessarily projective) Hyperk\"ahler manifold of dimension $2n$,
and denote by $q(.)$ its Beauville-Bogomolov form.
Denote by $\holom{\pi}{\PP(\Omega_X)}{X}$ the projectivisation of the cotangent bundle, and by $\zeta \rightarrow \PP(\Omega_X)$ the tautological class.
There exists a constant $C \geq 0$ depending only on the deformation family of $X$ such that
the following holds:
\begin{itemize}
\item  Let $\omega_X$ be a nef and big $(1,1)$-class on $X$ such that $q(\omega_X) \geq C$.
Then $\zeta+\pi^*\omega_X$ is pseudoeffective.
\item Suppose that $X$ is very general in its deformation space, and let 
$\omega_X$ be a nef and big $(1,1)$-class on $X$.
Then  $q(\omega_X) \geq C$ if and only if $\zeta+\pi^*\omega_X$ is nef.
\end{itemize}
\end{theorem}

The proof of the second statement is a combination of Demailly-P\v aun's criterion for nef cohomology classes with
classical results on the cohomology ring of very general Hyperk\"ahler manifolds: we show in Lemma \ref{lemmakeycomputation} that all the relevant intersection numbers are in fact polynomials in one variable,
the variable being the Beauville-Bogomolov form $q(\omega_X)$. The largest real roots of these polynomials turn out to be bounded from above, this yields the existence of the constant $C$. The first statement then follows by a folklore
degeneration argument that is proven by S. Diverio in the Appendix \ref{appendixdiverio}.

As an immediate consequence we obtain some good evidence for Conjecture \ref{conjecturegeneral}:

\begin{corollary} \label{corollarygoodevidence}
Let $X_0$ be a differentiable manifold of real dimension $4n$. Then there exist
at most finitely many deformation families of polarised Hyperk\"ahler manifolds $(X, H)$ 
such that $X_0 \stackrel{\mbox{\tiny diff.}}{\simeq} X$ and $H$ is an ample Cartier divisor on $X$ such that $\zeta +  \pi^* H$ 
is not pseudoeffective.
\end{corollary}

\subsection{Hyperk\"ahler manifolds of type $K3^{[n]}$}

While Theorem \ref{theoremmain} is quite satisfactory from a theoretical point of view, it 
it is not clear how to compute the constant $C$ in practice. We therefore prove a more explicit version
under a technical assumption:

\begin{theorem}\label{theoremHK}
Let $X$ be a (not necessarily projective) Hyperk\"ahler manifold of dimension $2n$. 
Suppose that a very general deformation of $X$ does not contain any proper subvarieties.
Let $\omega_X$ be a K\"ahler class on $X$. 
\begin{itemize}
\item  Suppose that
$$
(\zeta + \lambda \pi^* \omega_X)^{4n-1} > 0 \qquad \forall \ \lambda > 1.
$$
Then $\zeta+\pi^*\omega_X$ is pseudoeffective.
\item Suppose that $X$ is very general in its deformation space. Then 
$\zeta+\pi^*\omega_X$ is nef if and only if 
$$
(\zeta + \lambda \pi^* \omega_X)^{4n-1} > 0 \qquad \forall \ \lambda > 1.
$$
\end{itemize}
\end{theorem}

We also prove in Proposition \ref{propositionpseffcone} that for very general $X$, the class
$\zeta+\pi^*\omega_X$ is pseudoeffective if and only if it is nef. Thus Theorem \ref{theoremHK} is optimal at least
for very general $X$. Since $(\zeta + \lambda \pi^* \omega_X)^{4n-1}$ can be expressed as a polynomial depending only on the Segre classes of $X$, see equation \eqref{formulatop}, the sufficient condition
can be written down explicitly. 

If $\omega_X$ is the class of an ample divisor, the condition in 
Theorem \ref{theoremHK} essentially says that the leading term of the Hilbert polynomial
$$
\chi(\PP(\Omega_X), \sO_{\PP(\Omega_X)}(l(\zeta + \pi^* \omega_X)))
$$ 
is  positive. It is however possible that the higher cohomology of $\sO_{\PP(\Omega_X)}(l(\zeta + \pi^* \omega_X))$ grows with order $4n-1$, so it is not obvious
that $\zeta + \pi^* \omega_X$ is pseudoeffective.

Let $S$ be a K3 surface, and denote by $X:=S^{[n]}$ the Hilbert scheme parametrizing $0$-dimensional subschemes
of length $n$. Then $X$ is Hyperk\"ahler \cite{Bea83}, and by a theorem of Verbitsky 
\cite[Thm.1.1]{Ver98} a very general deformation does not contain any proper subvarieties.
Thus the technical condition in Theorem \ref{theoremHK} is satisfied for a Hyperk\"ahler manifold of deformation type $K3^{[n]}$. 
We compute the constant $C$ for Hilbert schemes of low dimension. In particular we obtain

\begin{corollary} \label{corollaryK3surface}
Let $S$ be a (not necessarily projective) K3 surface.
Let $\omega_S$ be a nef and big $(1,1)$-class on $S$ such that $\omega_S^2 \geq 8$. 
Then $\zeta+\pi^*\omega_S$ is pseudoeffective.
\end{corollary}

The theorem of Gounelas and Ottem shows that this result is optimal for infinitely many $19$-dimensional families of projective K3 surfaces. Their results also show that for certain families, e.g. general smooth quartics in $\PP^3$, our estimate is not optimal \cite[Cor.4.2]{GO18}. In these cases the obstruction comes from the projective geometry of $X$ \cite[Sect.4.2]{GO18}.

In higher dimension the situation becomes much more complicated. We show in 
Corollary \ref{corollaryK3square} that for a nef and big class $\omega_X$ on a Hilbert square $X:=S^{[2]}$
such that
$$
q(\omega_X)\geq 3+\sqrt{\frac{21}{5}},
$$
the class $\zeta+\pi^*\omega_X$ is pseudoeffective. This bound is optimal 
for a very general deformation of $X$. However a Hilbert square $S^{[2]}$ deforms as a complex manifold in a 
$21$-dimensional space, while its deformations as a Hilbert square only form a $20$-dimensional family.
In Section \ref{sectionK3square} we study in detail very general elements of the family of Hilbert squares:
since the Hilbert square always contains an exceptional divisor, it is obvious
that the nef cone and the pseudoeffective cone of $\PP(\Omega_{S^{[2]}})$ do not coincide. 
It is much more difficult to decide if $\zeta+\pi^* \omega_X$ is nef if it is pseudoeffective.
For this purpose we construct in Subsection \ref{subsectionZ} a ``universal''
subvariety $Z \subset \PP(\Omega_{S^{[2]}})$ that surjects onto $S^{[2]}$ and is an obstruction
to the nefness of $\zeta+\pi^* \omega_X$ (cf. Proposition \ref{propositionthenumber}).

{\bf Acknowledgements.}  This project was initiated during the first-named author's stay
in Nice, supported by the UCA JEDI project (ANR-15-IDEX-01). 
The second-named author thanks the Institut Universitaire de France and the A.N.R. project
project Foliage (ANR-16-CE40-0008) for providing excellent working conditions.
We thank D. Huybrechts for his remarks on the first version of this paper. 
We thank  S. Boissi\`ere and S. Diverio for very detailed communications, appearing in Proposition \ref{propositionK3cube}
and Appendix \ref{appendixdiverio} respectively.

\section{Notation and basic facts}

We work over $\C$, for general definitions we refer to \cite{Har77, Dem12}.
Manifolds and normal complex spaces will always be supposed to be irreducible.
We will not distinguish between an effective divisor and its first Chern class.

We recall some basic facts about the positivity of $(1,1)$-cohomology classes that generalise 
the corresponding notions for divisors classes.

\begin{definition}
Let $X$ be a compact K\"ahler manifold and $\alpha\in H^{1,1}(X, \R)$. The class $\alpha$ is a K\"ahler class if it can be represented by a smooth real form of type $(1,1)$ that is positive definite at every point. The class $\alpha$ is pseudoeffective if it can be represented by a closed real positive $(1,1)$-current. 
\end{definition}

The cone generated by the K\"ahler forms is the open convex cone $\mathcal{K}(X)$ in $H^{1,1}(X, \R)$ called K\"ahler cone. The cone generated by closed positive real $(1,1)$-currents is a closed convex cone denoted by $\mathcal{E}(X)$ called pseudoeffective cone. The closure of the K\"ahler cone is the nef cone and the interior of the pseudoeffective cone is the big cone. Clearly since a pseudoeffective class may be represented by a singular current the pseudoeffective cone contains the nef cone. 

Suppose now that $X$ is a projective manifold. Inside the real vector space $H^{1,1}(X, \R)$ there is the group of real divisors modulo numerical equivalence or real N\'eron--Severi space 
$$
\operatorname{NS}_\R(X)=\left( H^{1,1}(X, \R)\cap H^2(X,\Z) \right)\otimes_\Z \R
$$
Then we have
$$
\mathcal{K}(X)\cap \operatorname{NS}_\R(X) = \operatorname{Nef}(X),
\qquad
\mathcal{E}(X)\cap \operatorname{NS}_\R(X) =  \overline{\operatorname{Eff}(X)}
$$
where $\operatorname{Nef}(X)$ (resp. $\overline{\operatorname{Eff}(X)}$) is the nef cone (resp. pseudoeffective
cone) well-known to algebraic-geometers (cf. \cite{BDPP13} for more details).

In the analytic context it is difficult to characterise the positivity of a $(1,1)$-class via intersection numbers,
however we have the following easy consequence of the Demailly-P\v aun criterion \cite[Theorem 0.1]{DP04}:

\begin{lemma} \label{lemmaconnected}
Let $X$ be a compact K\"ahler manifold, and let $V$ be a vector bundle over $X$.
Denote by $\holom{\pi}{\PP(V)}{X}$ the natural morphism, and by 
$\zeta$ the tautological class on $\PP(V)$. Let $\omega_X$ be a K\"ahler class on $X$ such that
for all $\lambda \geq 1$ we have
$$
(\zeta+ \lambda \pi^* \omega_X)^{\dim Z} \cdot Z >0 \qquad \forall \ Z \subset \PP(V) \ \mbox{irreducible}.
$$
Then $\zeta+ \pi^* \omega_X$ is a K\"ahler class.
\end{lemma}

\begin{proof}
By assumption the class $\zeta+ \lambda \pi^* \omega_X$ is an element of the
positive cone $\mathcal P \subset H^{1,1}(\PP(V))$ of classes having positive intersection with all subvarieties.  By the Demailly-P\v aun criterion \cite[Theorem 0.1]{DP04} 
the K\"ahler cone $\mathcal K$ is a connected component of $\mathcal P$.
Since $\zeta$ is a relative K\"ahler class, we know that $(\zeta+ \lambda \pi^* \omega_X) \in \mathcal K$ for $\lambda \gg 0$ \cite[Proof of Prop.3.18]{Voi02}. Conclude by connectedness.
\end{proof}

A (not necessarily projective) Hyperk\"ahler manifold is a simply connected compact K\"ahler manifold $X$ such that $H^0(X, \Omega_X^2)$ is spanned by a symplectic form $\sigma$, \textsl{i.e.} an everywhere non-degenerate holomorphic two form.
The existence of the symplectic form $\sigma$ implies that $\dim X$ is even, so we will write $\dim(X)=2n$.
The symplectic form defines an isomorphism $T_X \rightarrow \Omega_X$, so the odd Chern classes of $X$ vanish.

The second cohomology group with integer coefficients $H^2(X,\Z)$ 
is a lattice for the Beauville-Bogomolov quadratic form $q=q_X$ \cite[Sect.8]{Bea83}. Somewhat abusively we denote by 
$q(. , .)$ the associated bilinear form.
If $X$ is projective it follows from the Bochner principle that all the symmetric powers $S^l \Omega_X$ are slope stable
with respect to any polarization $H$ on $X$ \cite[Thm.6]{Kob80}.

We will frequently use basic facts about the deformation theory of Hyperk\"ahler manifolds, as explained in
\cite[Sect.8]{Bea83} \cite[Sect.1]{Huy99}. In particular we use that a very general point of the deformation
space corresponds to a non-projective manifold, but the projective manifolds form a countable union of codimension 
one subvarieties that are dense in the deformation space. A very general deformation of $X$ is a manifold 
$X_t$ which corresponds to a very general point $t$ in the Kuranishi space of $X$.

The Picard group $\pic (X)$ is by definition the group of isomorphism classes of line bundles on $X$.
Since $H^1(X, \sO_X)=0$ and $H^2(X, \Z)$ is torsion-free, 
the Lefschetz $(1,1)$-theorem \cite[Prop.3.3.2]{Huy05} gives an isomorphism
$$
H^2(X, \Z) \cap H^{1,1}(X, \R) \simeq \pic (X).
$$

\begin{remark} \label{remarknosubvarieties}
	By Hodge theory a class $\alpha \in H^2(X, \Z )$ is of type $(1,1)$ if an only if it is orthogonal to the symplectic form $\sigma_X$. If $\sigma_X$ is not orthogonal to any non zero element of the lattice $H^2(X, \Z )$ then there are no integral cohomology classes of type $(1,1)$ in $X$. 
	For any $0\neq \alpha \in H^2(X, \Z )$ the orthogonal $\alpha^\bot \subset H^2(X, \C )$ is a proper hyperplane because the Beauville form $q$ is non degenerate. By local Torelli Theorem \cite[Th\'eor\`eme 5]{Bea83} the moduli space of the deformations of $X$ is locally an open inside the quadric $\{ q(\beta)=0\} \subset \PP (H^2(X,\C))$. So a very general Hyperk\"ahler manifold can be taken outside all the hyperplanes $\alpha^\bot $ such that $0\neq \alpha \in H^2(X, \Z)$, hence has trivial Picard group. 
\end{remark}

\begin{remark} \label{remarkEgeneral}
For any very general Hyperk\"ahler $X$ we have by \cite[Cor.1]{Huy03}
$$
\mathcal{E}^0(X)=\mathcal{K}(X) = \mathcal C (X)
$$
where  $\mathcal C (X)$ is the connected component of $\{ \alpha \in H^{1,1}(X, \R) \ | \ q(\alpha)>0 \}$ 
that contains  $\mathcal{K}(X)$.
In particular the classes in the boundary of the K\"ahler cone cannot be in the interior of the pseudoeffective cone because they have trivial top self intersection. Thus a big class, being in the interior of $\mathcal E(X)$,
is in fact K\"ahler. 
\end{remark}

Finally let us recall that for a vector bundle of rank $r$ over a compact K\"ahler manifold $M$, 
the $k$-th Segre class is defined as $\pi_* \zeta^{r+k}=(-1)^k s_k(V)$,
where $\holom{\pi}{\PP(V)}{M}$ is the projectivisation and $\zeta$ the tautological class.

\section{The projectivised cotangent bundle} \label{sectionsubvarieties}

Let $X$ be a compact K\"ahler manifold,
and let $V \rightarrow X$ be a vector bundle over $X$. Denote by $\zeta:=c_1(\o_V(1))$ the tautological
class on $\PP(V)$ and by $\holom{\pi}{\PP(V)}{X}$ the projection.
By \cite[Chapter 2]{Ko87} the cohomology ring with integral coefficients is 
$$
H^{\bullet}(\PP(V),\z)=H^{\bullet}(X, \z )[\zeta]/p(\zeta)
$$
where $p(\zeta)=\zeta^n+\zeta^{n-1} \pi^* c_1(V)+ \ldots + \pi^* c_{n}(V)$.	
	
Passing to  complex coefficients we get that any class $\alpha \in H^{2k}(\PP(V),\c) $ can be uniquely written as 
$$
\alpha= \sum_{p=0}^{k}  \zeta^p  \cdot \pi^* \beta_{2k-2p}
$$
where $\beta_{2k-2p} \in H^{2k-2p}(X, \c)$. 

Since $X$ is K\"ahler we can consider the Hodge decomposition
of $H^{2k}(\PP(V),\c)$
and obtain a decomposition 
$$
H^{k,k}(\PP(V))=\bigoplus_{i=0}^k \c \zeta^{k-i} \otimes \pi^* H^{i,i}(X).
$$
Using the canonical inclusion $H^{k,k}(\PP(V)) \subset H^{2k}(\PP(V),\c)$
we can compare the two decompositions and obtain 
\begin{equation} \label{decomposecohomology}
H^{k,k}(\PP(V))\cap H^{2k}(\PP(V),\z)=\bigoplus_{i=0}^k \z \zeta^{k-i}  \otimes \pi^* (H^{i,i}(X)\cap H^{2i}(X,\z ) ).
\end{equation}
In particular the cohomology class of a codimension $k$ subvariety $Z$ of $\PP(V)$ can be uniquely written as 
\begin{equation} \label{decompose}
[Z]=\beta_0  \zeta^k + \zeta^{k-1}\cdot \pi^* \beta_1 + \zeta^{k-2}\cdot \pi^* \beta_2 + \ldots + \pi^* \beta_k
\end{equation}
where $\beta_i\in H^{i,i}(X)\cap H^{2i}(X,\z )$ and $\beta_0 \in \z$.

In this section we will first use this decomposition to establish Theorem \ref{theoremmain}, see Subsection \ref{subsectionproofmain}. 
Then we will prove an additional restriction on the component $\beta_1$ that allows us
to describe the varieties $Z \subset \PP(\Omega_X)$ in some cases, see Subsection \ref{subsectionsubvarieties}.

\subsection{Proof of the main result} \label{subsectionproofmain}

It is well-known that the cohomology ring of a very general Hyperk\"ahler manifold $X$ is governed by its Beauville-Bogomolov form. We start by showing a similar property for the cohomology ring of $\PP(\Omega_X)$:

\begin{lemma}\label{lemmakeycomputation}
Let $X$ be a Hyperk\"ahler manifold of dimension $2n$, and denote by $q(.)$ its Beauville-Bogomolov form. 
Let
$$
\Theta \in 
H^{k,k}(\PP(\Omega_X))\cap H^{2k}(\PP(\Omega_X),\z)
$$
be an integral class of type $(k,k)$. Suppose that the class $\Theta$ is of type $(k,k)$ for every small deformation
of $X$.
Then there exists a polynomial 
$p_\Theta(t) \in \Q[t]$ such that for any $(1,1)$-class $\omega$ on $X$, one has
$$
(\zeta + \pi^* \omega)^{4n-1-k} \cdot \Theta = p_\Theta(q(\omega)).
$$
\end{lemma}

\begin{proof}
Observe first that both sides of the equation are polynomial functions on $H^{1,1}(X)$.
In particular they are determined by their values on an open set and we can assume 
without loss of generality that $\omega$ is K\"ahler.  
Let 
\begin{equation} \label{decZ}
\Theta=\sum_{i=0}^{k} \zeta^{k-i} \pi^* \beta_{i} 
\end{equation}
be the decomposition of $\Theta$ according to  \eqref{decomposecohomology} 
where $\beta_{i} \in H^{i,i} (X) \cap H^{2i}(X, \Z)$.  
By our assumption, for any small deformation $\mathfrak X \rightarrow \Delta$, the class $\Theta$ deforms
as an integral class $\Theta_t$ of type $(k,k)$. Thus we can write
$$
\Theta_t=\sum_{i=0}^{k} \zeta_t^{k-i} \pi^* \beta_{i,t} 
$$
with $\beta_{i,t} \in H^{i,i} (X_t) \cap H^{2i}(X_t, \Z)$. Since the family $\PP(\mathfrak X) \rightarrow \Delta$
is locally trivial in the differentiable category, we can consider the classes $\beta_{i}$ as elements of
$H^{2i}(X_t, \Z)$ for $t \neq 0$.  The integral cohomology class $\Theta_t \in H^{2k}(\PP(\Omega_{\mathfrak X_t}, \Z)$ does not depend on $t$, so \eqref{decZ} induces a decomposition
$$
\Theta_t=\sum_{i=0}^{k} \zeta_t^{k-i} \pi^* \beta_{i}.
$$
By uniqueness of the decomposition we have $\beta_i = \beta_{i,t}$, in particular the classes
$\beta_{i}$ are of type $(i,i)$ in $\mathfrak X_t$.

We have 
$$
(\zeta+\pi^*\omega)^{4n-1-k}
= 
 \sum_{j=0}^{4n-1-k} \binom{4n-1-k}{j}  \zeta^{4n-1-k-j} \pi^* \omega^j,
$$
so 
$$
(\zeta+\pi^*\omega)^{4n-1-k} \cdot \Theta
=
\sum_{j=0}^{4n-1-k} 
\binom{4n-1-k}{j}  
\sum_{i=0}^{k} \zeta^{4n-1-j-i} \pi^* (\beta_{i} \cdot \omega^j)
.
$$
By the projection formula and the definition of Segre classes one has for $i+j \leq 2n$
$$
\zeta^{4n-1-j-i} \pi^* (\beta_{i} \cdot \omega^j)
=
(-1)^{i+j} s_{2n-j-i} \cdot \beta_{i} \cdot \omega^j.
$$
Since the odd Segre classes of a Hyperk\"ahler manifold vanish, we can implicitly assume that
$i+j$ is even. In particular $(-1)^{i+j}=1$.
We claim that we can also assume that $j$ is even.

{\em Proof of the claim.}
Note that $f(\omega):=s_{2n-j-i} \cdot \beta_{i} \cdot \omega^j$ defines a polynomial on $H^{1,1}(X)$.
Thus, up to replacing $\omega$ by a general K\"ahler class, we can assume that $s_{2n-j-i} \cdot \beta_{i} \cdot \omega^j = 0$
if and only if $s_{2n-j-i} \cdot \beta_{i} \cdot (\omega')^j = 0$ for every $(1,1)$-class $\omega'$. As we have already observed at the start of the proof, we can make this generality assumption without loss of generality.
If $s_{2n-j-i} \cdot \beta_{i} \cdot \omega^j = 0$, the term is irrelevant for our computation.
If $s_{2n-j-i} \cdot \beta_{i} \cdot \omega^j \neq 0$,
then by \cite[Thm.2.1]{Ver96} the degree of the cohomology class $s_{2n-j-i} \cdot \beta_{i}$ is divisible by $4$ (here we use that $\omega$ is a K\"ahler class).
Since $s_{2n-j-i} \cdot \beta_{i} \in H^{4n-2j}(X, \R)$, the claim follows.

Thus we obtain
$$
(\zeta+\pi^*\omega)^{4n-1-k} \cdot \Theta
=
\sum_{j=0}^{4n-1-k} 
\binom{4n-1-k}{j}  
\sum_{i=0}^{k} s_{2n-j-i} \cdot \beta_{i} \cdot \omega^j
.
$$
We have shown above that the classes $s_{2n-j-i} \cdot \beta_{i}$ 
are of type $(2n-j, 2n-j)$ on all small deformations of $X$.
Since $j$ is even, we know by \cite[Theorem 5.12]{Huy97} that there exist 
constants $d_{i,j} \in \Q$ such that for any 
$\delta\in H^{1,1}(X,\R)$ we have 
$$
s_{2n-j-i} \cdot \beta_{i} \cdot \delta^j = d_{i,j} q(\delta)^{j/2}
$$
The polynomial 
$$
p_\Theta(t) : = \sum_{j=0}^{4n-1-k} 
\binom{4n-1-k}{j}  
\sum_{i=0}^{k} d_{i,j} t^{j/2}
$$
has the claimed property.
\end{proof}

\begin{proof}[Proof of Theorem \ref{theoremmain}]
Suppose first that $X$ is very general in its deformation space. 
Let $Z \subset \PP(\Omega_X)$ be a subvariety.
Since $X$ is very general, we know that 
for any small deformation $\mathfrak X \rightarrow \Delta$, the variety $Z$ deforms
to a variety $Z_t \subset \PP(\Omega_{\mathfrak X_t})$.
In particular its cohomology class $[Z]$ is of type $(k,k)$ for every small deformation.
Thus Lemma \ref{lemmakeycomputation} applies and there exists a polynomial $p_Z(t)=p_{[Z]}(t)$ such that
$$
(\zeta + \pi^* \omega)^{4n-1-k} \cdot [Z] = p_Z(q(\omega))
$$
for any $(1,1)$-class $\omega$ on $X$. Since intersection numbers are invariant under deformation and the cycle space
has only countably irreducible components, we obtain a countable number of polynomials
$(p_m(t))_{m \in \N}$ such that for every subvariety $Z \subset \PP(\Omega_X)$ there exists a polynomial $p_m$
such that
$$
(\zeta + \pi^* \omega)^{4n-1-k} \cdot [Z] = p_m(q(\omega)).
$$
Denote by $c_m$ the largest real root of the polynomial $p_m$. We claim that 
$$
\sup_{m \in \N} \{ c_m \} < \infty.
$$
Indeed fix a K\"ahler class $\eta$ on $X$ such that $\zeta+\pi^*\eta$ is a K\"ahler class on $\PP(\Omega_X)$. Then $\zeta+ \lambda \pi^*\eta$ is a K\"ahler class for all $\lambda \geq 1$, so
$$
p_m(\lambda^2 q(\eta)) = (\zeta + \lambda \pi^* \eta)^{4n-1-k} \cdot [Z] > 0
$$
for all $\lambda \geq 1$. In particular $c_m \leq q(\eta)$, and hence $\sup_{m \in \N} \{ c_m \} \leq q(\eta)$.
This shows the claim and we denote the real number $\sup_{m \in \N} \{ c_m \}$ by $C$.

{\em Proof of the second statement.} 
Since $X$ is very general, we know by Remark \ref{remarkEgeneral} that the nef and big class $\omega_X$ is K\"ahler.
If $q(\omega_X) > C$ then by construction of the constant $C$ one has
$$
(\zeta + \lambda \pi^* \omega_X)^{4n-1-k} \cdot [Z] = p_m(\lambda^2 q(\omega_X)) > 0
$$
for every subvariety $Z$. By Lemma \ref{lemmaconnected} this implies that $\zeta+\pi^* \omega_X$ is K\"ahler.
If $q(\omega_X) \geq C$ then $q((1+\varepsilon) \omega_X) > C$, so $\zeta+(1+\varepsilon) \pi^* \omega_X$ is K\"ahler.
Thus $\zeta+\pi^* \omega_X$ is nef. 

Vice versa suppose that $\zeta+\pi^* \omega_X$ is nef. Then $\zeta+ \lambda \pi^* \omega_X$ is nef for all $\lambda \geq 1$.
Thus
$$
p_m(\lambda^2 q(\omega_X)) = (\zeta + \lambda \pi^* \omega_X)^{4n-1-k} \cdot [Z] \geq 0
$$
for all $\lambda \geq 1$. Since $\lim_{\lambda \to \infty} \lambda^2 q(\omega_X) = \infty$, this implies $c_m \leq q(\omega_X)$ for all $m \in \N$. Hence we obtain $q(\omega_X) \geq C$.

{\em Proof of the first statement.}
We claim that we can assume that $\omega_X$ is a K\"ahler class with $q(\omega_X)>C$. 
Indeed let $\delta$ be any K\"ahler class on $X$, then $\omega_X+\delta$ is K\"ahler. Moreover one has
$$
q(\omega_X+\delta)=q(\omega_X)+q(\delta)+2q(\delta,\omega)>
q(\omega_X) \geq C
$$
Thus if $\zeta+\pi^* (\omega_{X}+\delta)$ is pseudoeffective for every $\delta$, then
the closedness of the pseudoeffective cone implies the statement by
taking the limit $\delta \to 0$. This shows the claim.

We denote by $0\in \operatorname{Def}(X)$ the point corresponding to $X$ in its Kuranishi family. By \cite[Proposition 5.6]{Huy16} we can assume that in a neighborhood $U$ of $0\in \operatorname{Def}(X)$ the K\"ahler class $\omega_X$ deforms as a K\"ahler class $(\omega_{X_t})_{t \in U}$. 
In order to  simplify the notation we replace $U$ with a very general disc $\Delta$ centered at $0$ and consider the family $\mathcal{X} \rightarrow \Delta$. 
Since the Beauville--Bogomolov form is continuous we have, up to replacing $\Delta$ by a smaller disc, that
 $q(\omega_{X_t}) > C$ for every $t \in \Delta$.
By the second statement this implies that for $t \in \Delta$ very general the class $\zeta_t + \pi_t^* \omega_{X_t}$ is nef, in particular it is pseudoeffective. Now we apply Theorem \ref{limitpseff} to the family $\PP(\Omega_\mathfrak{X}) \rightarrow \Delta$ and the classes $\zeta_{t}+\pi^*\omega_{X_t}$: this shows that $\zeta + \pi^* \omega_{X}$
is pseudoeffective.
\end{proof}

\begin{proof}[Proof of Corollary \ref{corollarygoodevidence}]
By \cite[Theorem 2.1]{Huy03b} there exist at most finitely many
different deformation families of irreducible holomorphic symplectic complex structures on $X_0$. 
For any such deformation type, Theorem \ref{theoremmain} gives a constant $C_k$ such that $\zeta + \pi^* \omega_{X}$ is pseudoeffective for every K\"ahler class $\omega_X$ such that $q(\omega_{X})>C_k$. Let $C$ be the maximum among the constants $C_k$. Since the differentiable structure on $X$ is fixed, the constant of proportionality between the Beauville--Fujiki form $q(\omega_X)$ and the top intersection $\omega_X^{2n}$ is fixed. 
Thus the polarised Hyperk\"ahler manifolds $(X, H)$ such that 
$X_0 \stackrel{\mbox{\tiny diff.}}{\simeq} X$ and $\zeta + \pi^* H$ is not pseudoeffective satisfy $H^{2n} \leq b$ 
for some constant $b$. By a theorem of Matsusaka-Mumford \cite{MM64} 
there are for any fixed $0 < i\leq b$ only a finite number of deformation families
of polarised Hyperk\"ahler manifolds $(X, H)$ such that $H^{2n}=i$.
Thus the cases where $\zeta + \pi^* H$ is not pseudoeffective belong to one of these finitely many families.
\end{proof}

\subsection{Subvarieties of the projectivised cotangent bundle} \label{subsectionsubvarieties}

We start with a technical observation:

\begin{lemma} \label{lemmakey}
Let $X$ be a projective Hyperk\"ahler manifold  of dimension $2n$. 
Let $Z$ be an effective cycle on $\PP(\Omega_X)$ 
of codimension $k>0$ such that $\pi(\supp Z) = X$. Denote by
$$
[Z]= \beta_0  \zeta^k + \zeta^{k-1}\cdot \pi^* \beta_1 + \zeta^{k-2}\cdot \pi^* \beta_2 + \ldots + \pi^* \beta_k
$$
the decomposition \eqref{decompose} of its cohomology class. Then we have $\beta_1 \neq 0$.
\end{lemma}

\begin{proof}
We argue by contradiction and suppose that $\beta_1=0$. Let $C \subset X$ be a general complete intersection
of sufficiently ample divisors $D_i \in | H |$ so that the Mehta--Ramanathan theorem \cite[Thm.4.3]{MR84} applies for $\Omega_X$. Then the restriction $\Omega_X|_C$ is stable, and by a result of Balaji and Koll\'ar \cite[Prop.10]{BK08}  
its algebraic holonomy group is $\mbox{Sp}_{2n}(\C)$. Thus not only $\Omega_X|_C$, but also all its symmetric
powers $S^l \Omega_X|_C$ are stable. Denote by $Z_C$ the restriction of the effective cycle $Z$ 
to $\PP(\Omega_X|_C)$. 
Since $\pi(\supp Z) = X$ the effective cycle $Z_C$ is not zero.
Then its cohomology class is
$$
[Z_C]= (\beta_0  \zeta^k + \zeta^{k-2}\cdot \pi^* \beta_2 + \ldots + \pi^* \beta_k) \cdot \pi^* H^{2n-1}
= \beta_0 \zeta_C^k
$$
where $\zeta_C$ is the restriction of the tautological class. In particular, since $c_1(\Omega_X|_C)=0$, we 
have $\zeta_C^{2n-k} \cdot [Z_C] = \beta_0 \zeta_C^{2n} = 0$. Yet this is a contradiction to \cite[Prop.1.3]{HP19}.
\end{proof}

\begin{remark}
Lemma \ref{lemmakey} also holds if $X$ is a Calabi-Yau manifold (in the sense of \cite{Bea83}): the cotangent bundle
$\Omega_X$ is also stable and the algebraic holonomy is $\mbox{SL}_{\dim X}(\C)$ \cite[Prop.10]{BK08}.
Thus the proof above applies without changes. 
\end{remark}

In \cite[Cor.2.6]{COP10} it is shown that a very general Hyperk\"ahler manifold is not {\em covered} by proper subvarieties. We show an analogue for the projectivised cotangent bundle $\Omega_X$:

\begin{lemma}\label{lemmanosubvariety}
Let $X$ be a Hyperk\"ahler manifold of dimension $2n$. Suppose that $X$ is very general in the following sense: 
we have
\begin{enumerate}
\item $\pic (X)=0$;
\item if $\mathfrak X \rightarrow \Delta$ is a deformation of $X=\mathfrak X_0$, then every irreducible component of the
cycle space $\chow{\PP(\Omega_{\mathfrak X_0})}$ deforms to $\chow{\PP(\Omega_{\mathfrak X_t})}$
for $t \neq 0$.
\end{enumerate}
Let $Z \subsetneq \PP(\Omega_X)$ be a  compact analytic subvariety.
Then $\pi(Z) \subsetneq X$.
\end{lemma}

By countability of the irreducible components of the relative cycle space \cite[Thm.]{Fuj79} and by Remark \ref{remarknosubvarieties} we know that for a very general choice of $X$ the hypothesis of the lemma are satisfied. 

\begin{proof}
We argue by contradiction, and suppose that $Z$ is a subvariety of $\PP(\Omega_X)$ of codimension $k>0$ such that $\pi(Z) = X$.
Denote by 
$$
[Z]=\beta_0  \zeta^k + \zeta^{k-1}\cdot \pi^* \beta_1 + \zeta^{k-2}\cdot \pi^* \beta_2 + \ldots + \pi^* \beta_k
$$
the decomposition \eqref{decompose} of its cohomology class. Since $\pic (X)=0$ we know that 
$\beta_1=0$.
 
Projective Hyperk\"ahler manifolds are dense in the deformation space of any Hyperk\"ahler manifold \cite[Sect.9]{Bea83} \cite[Prop.5]{Buc08},
so we can consider a small deformation of $X$
	\[ \xymatrix{
		X\ar[r] \ar[d] & \mathfrak{X}\ar[d]\\
		0 \ar[r]& \Delta 
	}\]
such that $\mathfrak X_{t_0}$ is projective for some point $t_0 \in \Delta$.
This deformation comes naturally with a deformation of the cotangent bundle, so we have a diagram
	\[ \xymatrix{
		\PP(\Omega_X) \ar[r]\ar[d]^{\pi} & \PP(\Omega_{\mathfrak{X}/\Delta})\ar[d]\\
		X\ar[r] \ar[d] & \mathfrak{X}\ar[d]\\
		0 \ar[r]& \Delta 
	}\]
By the second assumption the subvariety $Z\subset \PP(\Omega_X)$ deforms 
in a family of subvarieties $Z_{t}\subset \PP(\Omega_{X_{t}})$
having cohomology class 
$$
[Z_t]=\beta_0  \zeta^k + \zeta^{k-2}\cdot \pi^* \beta_2 + \ldots + \pi^* \beta_k.
$$
Since the cycle space is proper over the base $\Delta$ \cite[Th\'eor\`eme 1]{Ba75} we obtain in particular that the class
$\beta_0  \zeta^k + \zeta^{k-2}\cdot \pi^* \beta_2 + \ldots + \pi^* \beta_k$ is effectively represented
on $\PP(\Omega_{X_{t_0}})$. This contradicts Lemma \ref{lemmakey}.
\end{proof}

\begin{corollary}\label{corollarynosubvariety}
Let $X$ be a Hyperk\"ahler manifold of dimension $2n$. Suppose that $X$ is very general in the  sense of Lemma \ref{lemmanosubvariety}. 
Suppose also that
$X$ contains no proper compact subvarieties.
Let $Z \subsetneq \PP(\Omega_X)$ be a  compact analytic subvariety.
Then $\pi(Z)$ is a point.
\end{corollary}

\begin{proof}
By Lemma \ref{lemmanosubvariety} we have $\pi(Z) \subsetneq X$
for every subvariety $Z \subsetneq \PP(\Omega_X)$.
By our assumption this implies that $\pi(Z)$ is a point.
\end{proof}

\begin{remark} \label{remarknosubvarieties2}
A very general deformation of Kummer type
does {\em not} satisfy the assumptions of the corollary (\cite[Sect.6.1]{KV98I})
\end{remark}

\section{The positivity threshold}
\label{sectionpositivity}

In view of the results from Subsection \ref{subsectionsubvarieties}, we will deduce
Theorem \ref{theoremHK} from the main result:

\begin{proposition}\label{propositionsegre}
Let $X$ be a Hyperk\"ahler manifold of dimension $2n$. Suppose that
a very general deformation of $X$ contains no proper compact subvarieties.
Let $p_X(t)$ be the polynomial defined by applying Lemma \ref{lemmakeycomputation}
to $[\PP(\Omega_X)]$. 
Then the constant $C$ appearing in Theorem \ref{theoremmain} 
is the largest real root of $p_X(t)$.
\end{proposition}	

\begin{proof}
Since $C$ only depends on the deformation family we can assume that $X$ is very general in its deformation space.
In the proof of Theorem \ref{theoremmain} we defined the constant $C$ as 
$\sup_{m \in \N} \{ c_m \}$ where $c_m$ is the largest real root 
of the polynomials $p_m(t)$,
and the family of polynomials $(p_m(t))_{m \in \N}$ is obtained by applying 
Lemma \ref{lemmakeycomputation} to the classes of {\em all} the subvarieties $Z \subset \PP(\Omega_X)$.

By our assumption and Corollary \ref{corollarynosubvariety} we know that a proper subvariety $Z \subsetneq \PP(\Omega_X)$
is contained in a fibre. Thus for any K\"ahler class $\omega_X$ the restriction
$$
(\zeta+\pi^* \omega_X)_Z = \zeta|_Z = c_1(\sO_{\PP^{2n-1}}(1))|_Z
$$
is ample. Hence the corresponding polynomial $p_m(t)$ is constant and positive. In particular there is no real root to take into account for the supremum. 
\end{proof} 

\begin{proof}[Proof of Theorem \ref{theoremHK}]
By Proposition \ref{propositionsegre} the constant $C$ in Theorem \ref{theoremmain} is the largest real root
of the polynomial $p_X(t)$ defined by
$$
p_X(q(\omega)) = (\zeta + \pi^* \omega)^{4n-1}.
$$
Thus the condition $q(\omega_X) \geq C$ is equivalent to 
$$
(\zeta + \lambda \pi^* \omega)^{4n-1} > 0
$$
for all $\lambda > 1$. Conclude with Theorem \ref{theoremmain}.
\end{proof}

We have already observed that for a very general Hyperk\"ahler manifold
the pseudoeffective cone and the nef cone coincide. 
This also holds for the projectivised cotangent bundle:

\begin{proposition} \label{propositionpseffcone}
Let $X$ be a Hyperk\"ahler manifold of dimension $2n$. Suppose that $X$ is very general in the  sense of Lemma \ref{lemmanosubvariety}. 
Suppose also that
$X$ contains no proper compact subvarieties.

Let $C \geq 0$ be the constant from Theorem \ref{theoremmain}.
Then we have
\begin{equation} \label{formulapseffcone}
	\mathcal{E}(\PP(\Omega_X^1))=\{ a\zeta+\pi^* \delta |\ a\geq 0, \delta\in \overline{\mathcal{K}(X)},\  q(\delta)\geq a^2 C  \}
\end{equation}
and	
$$
\mathcal{E}(\PP(\Omega_X^1))=\overline{\mathcal{K}(\PP(\Omega_X^1))}.
$$
\end{proposition}

\begin{proof}
	We start proving the last statement.
	We recall the definition of the \emph{Null cone} of $\PP(\Omega_X^1)$ that is the following set
	$$ 
	\mathcal{N}:=\{x\in H^{1,1} (\PP(\Omega_X^1),\R) \ | \ \int_{\PP(\Omega_X^1)}x^{2n-1}=0 \}. $$ 
	For any class $\gamma \in \partial \mathcal{K}(\PP(\Omega_X^1))$ there exists a subvariety $V$ of $\PP(\Omega_X^1)$ such that $\int_V \gamma^{\operatorname{dim(V)}}=0$. Since we are assuming that there are no proper subvarieties in $X$, by Lemma \ref{lemmanosubvariety} we know that the proper subvarieties of $\PP(\Omega_X^1)$ are contracted to points in $X$. Since $\PP(\Omega_X^1)$ is a projective bundle the integral along a contracted subvariety $V$ has the following property $$\int_V (a \zeta+\pi^*\delta)^{\operatorname{dim}(V)}=0\Leftrightarrow a=0.$$
	This implies using \cite[Theorem 0.1]{DP04} that $$\partial \mathcal{K}(\PP(\Omega_X^1))\subseteq \mathcal{N} \cup \{a=0\}. $$
	A $(1,1)$ form in the hyperplane $\{a=0\}$ is in the null cone. This tells that the K\"ahler cone is one of the connected component of $H^{1,1}(\PP(\Omega_X^1),\R)\setminus \mathcal{N}$. Hence the classes in the boundary of the K\"ahler cone are nef classes with trivial self intersection, so they are also in the boundary of the pseudoeffective cone \cite[Thm.0.5]{DP04}. This proves that the closure of the K\"ahler cone is the pseudoeffective cone.

	For notation's convenience we call $\mathcal{A}:=\{ a\zeta+\pi^* \delta |\ a\geq 0, \delta\in \overline{\mathcal{K}(X)},\  q(\delta)\geq a^2 C  \}$.
	The inclusion $\mathcal{E}(\PP(\Omega_X^1)) \supseteq \mathcal{A}$ follows from the first
	statement of Theorem \ref{theoremmain}. 
	To prove the other inclusion we argue as follows. The points of $\partial \mathcal{A}$ are contained in the set $\{a=0 \vee q(\delta)= a^2 C \}$. By definition of the constant $C$ the self intersection of the classes $a \zeta+\pi^*\delta$ vanishes. We also have $(\pi^*\delta )^{2n-1}=0$, hence $$\partial \mathcal{A} \subset \mathcal{N}.$$ 
	Moreover there are no points in the interior of $\mathcal{A} $ contained in the null cone, so $\mathcal{A}^{\circ} $ must be a connected component of $H^{1,1}(\PP(\Omega_X^1),\R)\setminus \mathcal{N}$. Since the intersection of $\mathcal{E}(\PP(\Omega_X^1))$ and $ \mathcal{A}$ is non-empty and both are closed convex cones the conclusion follows.
\end{proof}

\begin{remark}
The rest of the paper is devoted to giving more explicit expressions of the conditions 
in Theorem \ref{theoremmain} and Theorem \ref{theoremHK}, so for clarity's sake let us write down the polynomial 
$p_X(t)$ from Proposition \ref{propositionsegre}:
let $X$ be a Hyperk\"ahler manifold of dimension $2n$, and denote by $\zeta$ the tautological class of $\pi: \PP(\Omega_X)\rightarrow X$.
Recall that by definition of the Segre classes  we have $\pi_* \zeta^{2n+i}=(-1)^i s_i(X)$.
Since the odd Chern classes of a Hyperk\"ahler manifold are trivial, 
the odd Segre classes vanish. Note also that $(\pi^*\omega_X)^i=0$ if $i > 2n$.
The top self-intersection is thus 

\begin{equation}\label{formulatop}
  \begin{split}
p_X(\lambda q(\omega_X))=  (\zeta+\lambda \pi^*\omega_X)^{4n-1}=\sum_{i=0}^{2n} \binom{4n-1}{i}\zeta^{4n-1-i} \cdot \pi^*\omega_X^{i} \lambda^i 
    \\ = \zeta^{2n-1}\sum_{i=0}^{n} \binom{4n-1}{2i} \zeta^{2n-2i} \cdot \pi^* \omega_X^{2i} \lambda^{2i} 
\\ =\sum_{i=0}^{n} \binom{4n-1}{2i} s_{2n-2i}(X) \cdot \omega_X^{2i} \lambda^{2i}. 
  \end{split}
\end{equation}
Recall also 
that by  \cite[Remark 4.12]{F87} 
there exist constants  $d_{2i} \in \R$
that depend only on the family such that 
\begin{equation}\label{intersectionsegre}
s_{2n-2i}(X) \cdot \omega_X^{2i}= d_{2i} q(\omega_X)^{i}
 \end{equation}
for any $(1,1)$-class $\omega_X$. Note that $s_{0}(X) \cdot \omega_X^{2n}=  \omega_X^{2n} = d_{2n} q(\omega_X)^{n}$, so $d_{2n}>0$.
\end{remark}

\begin{example} \label{exampletop}
For $n=1$ we obtain
$$
(\zeta+\lambda \pi^*\omega_X)^{3} = - c_2(X)+3 \omega_X^2 \lambda^2.
$$
For $n=2$ we obtain
$$
(\zeta+ \lambda\pi^* \omega_X)^7
=(c_2(X)^2-c_4(X))-21 c_2(X) \cdot \omega_X^2 \lambda^2+ 35 \omega_X^4 \lambda^4.
$$
\end{example}

\begin{proof}[Proof of Corollary \ref{corollaryK3surface}]
By Proposition \ref{propositionsegre} we only have to compute the
largest real root of $p_X(t)$.  
By Formula \eqref{formulatop} and
Example \ref{exampletop} the constant $C$ is the largest root of $-c_2(X)+3 t=0$. Since $c_2(X)=24$ the result follows.
\end{proof}

\begin{definition}
Let $X$ be a Hyperk\"ahler manifold of dimension $2n$, and
let $\omega_X$ be a nef and big class on $X$. 
The \emph{positivity threshold}  of $(X, \omega_X)$ is defined as
$$
\gamma_p(\omega_X):=\inf \{ \lambda_0 \in \r | \ (\zeta+ \lambda \pi^*\omega_X)^{4n-1}>0 \qquad \forall \lambda > \lambda_0 \}.
$$
\end{definition}

\begin{remark}\label{solution}
Since $(\zeta+ \lambda\pi^*\omega_X)^{4n-1}\sim \lambda ^{2n}\omega_X^{2n}$ for $t \gg 0$
we have $\gamma_p(\omega_X)<+\infty$. 
It seems unlikely that $(\zeta+ \lambda \pi^*\omega_X)^{4n-1}>0$ for all $\lambda \in \R$. If
(a very general deformation of) $X$ contains no proper subvarieties, this can be seen as follows:
since $X$ has no subvarieties, the nef and big class $\omega_X$ is K\"ahler.
By Corollary \ref{corollarynosubvariety}, the class $\zeta+ \lambda \pi^*\omega_X$ satisfies the condition
of Lemma \ref{lemmaconnected} for any $\lambda \in \R$, so  $\zeta+ \lambda \pi^*\omega_X$ is K\"ahler
for any $\lambda \in \R$. But $\mathcal K(\PP(\Omega_X))$ does not contain any lines.
\end{remark}

Let $X$ be a Hyperk\"ahler manifold, and let $\omega_X$ be a K\"ahler class on $X$.
We  define the \emph{pseudoeffective threshold}
$$
\gamma_e(\omega_X):=\inf \{ t\in \r | \ \zeta+t\pi^*\omega_X \ \text{is big/pseudoeffective}\}
$$
and the \emph{nef threshold}
$$
\gamma_n(\omega_X):=\inf \{ t\in \r | \ \zeta+t\pi^*\omega_X \  \text{is K\"ahler/nef}\}.
$$
Since $\zeta+t\pi^*\omega_X$ is K\"ahler for $t \gg 0$, both thresholds are real numbers.

\begin{proposition}\label{threshold}
Let $X$ be a (not necessarily projective) Hyperk\"ahler manifold of dimension $2n$. 
Suppose that a very general deformation of $X$ does not contain any proper subvarieties.
Let $\omega_X$ be a K\"ahler class on $X$.
Then we have
$$
\gamma_e(\omega_X) \leq \gamma_p(\omega_X)  \leq \gamma_n(\omega_X).
$$ 
For a very general deformation of $X$ these inequalities are equalities for any K\"ahler class $\omega_X$.
\end{proposition}

\begin{proof}
The top self-intersection of a K\"ahler class is certainly positive, so the inequality
$\gamma_p(\omega_X)  \leq \gamma_n(\omega_X)$ is trivial.
The inequality $\gamma_e (\omega_X) \leq \gamma_p(\omega_X)$
follows from Theorem \ref{theoremHK}.
For a very general deformation of $X$ we can apply Proposition \ref{propositionpseffcone},
so the nef cone and the pseudoeffective cone coincide. Thus we have $\gamma_e(\omega_X)  = \gamma_n(\omega_X)$.
\end{proof}

We will show in Section \ref{sectionK3square} that for the Hilbert square of a K3 surface the second inequality is strict.

\section{Hilbert square of a K3 surface} \label{sectionK3square}

\subsection{Setup}
\label{subsectionsetup}

We recall the basic geometry of the Hilbert square, using the notation and results of \cite[Sect.6]{Bea83}:
let $S$ be a (not necessarily algebraic) K3 surface, and let $\holom{\rho}{\widetilde{S \times S}}{S \times S}$ be the blow-up along the diagonal $\Delta \subset S \times S$. We denote the exceptional divisor of this blowup by $E$.
The natural involution on the product $S \times S$ lifts to an involution 
$$
i_{\blS} : \blS \rightarrow \blS,
$$ 
and we denote by
$\holom{\eta}{\blS}{X}$ the ramified two-to-one covering defined by taking the quotient with respect to this involution.
It is well-known that $X$ is smooth and Hyperk\"ahler. 
Finally we denote by $\holom{\pi}{\PP(\Omega_X)}{X}$ the natural projection, and by $\zeta \rightarrow \PP(\Omega_X)$
the tautological divisor.

Recall  that $X$ is isomorphic to the Hilbert scheme of length two zero dimensional subschemes $S^{[2]}$, and denote by
$$
\holom{\varepsilon}{S^{[2]}}{S^{(2)}}
$$
the natural map to the symmetric product. We denote by $E_X \subset X$ the exceptional
divisor of this contraction, and observe that $\eta|_E$ induces an isomorphism $E \simeq E_X$.  
Since $\rho$ is the blowup of the diagonal one has
$E \simeq \PP(\Omega_S)$, and we denote by
$$
\pi_S := \rho|_E \simeq \eta|_{E_X} : \PP(\Omega_S) \rightarrow S 
$$
the natural map. Denote by $\zeta_S \rightarrow \PP(\Omega_S)$ the tautological divisor.

By \cite[Sect.6,Prop.6]{Bea83} we have a canonical inclusion
$i : H^2(S, \Z) \hookrightarrow H^2(X, \Z)$  inducing a morphism of Hodge structures
$$
H^2(X, \Z) \simeq H^2(S, \Z) \oplus \Z \delta
$$
where $\delta$ is a primitive class such that $2 \delta = E_X$.
This decomposition is orthogonal with respect to the Beauville--Bogomolov quadratic form $q$ \cite[Sect.9, Lemma 1]{Bea83}
and one has $q(\delta)=-2$ \cite[Sect.1, Rque.1]{Bea83}. 
By construction of the inclusion $i$ \cite[Sect.6, Prop.6]{Bea83}
we have
\begin{equation} \label{restrictalphaX}
\alpha_X|_{E_X} = 2 \pi_S^* \alpha_S,
\end{equation}
and by \cite[Sect.9, Rque. 1]{Bea83} one has $q(\alpha_X)=\alpha_S^2$.

Since $E$ is the ramification divisor of the two-to-one cover $\eta$, we have $\eta^* E_X = 2 E$.
Since $E|_E = -\zeta_S$ and $2 \delta = E$, we obtain
\begin{equation} \label{restrictdelta}
\delta|_{E_X} = - \zeta_S.
\end{equation}

By \cite[Sect.9, Lemma 1]{Bea83} we have
\begin{equation} \label{qsurface}
\alpha^4 = 3 q(\alpha)^2
\end{equation}
for any $\alpha\in H^{1,1} (X)$. If $\alpha_S$ is any $(1,1)$-class on $S$, we set $\alpha_X := (i \otimes \id_\C)(\alpha_S)$.

The second Chern class $c_2(X)$ is a multiple of the Beauville--Bogomolov form. More precisely we have
\begin{equation} \label{qctwo}
c_2(X) \cdot \alpha^2 = 30 q(\alpha)
\end{equation}
for any $\alpha\in H^{1,1} (X)$ \cite[Section 3.1]{Ott15}.

\subsection{Intersection computation on $X$}

Denote by $\holom{p_i}{S \times S}{S}$ the projection on the $i$-th factor. 
The composition of $p_i$ with the blow-up $\rho$ defines a submersion
$$
p_i \circ \rho : \blS \rightarrow S,
$$
the fibre over a point $x \in S$ being isomorphic to the blow-up of $S$ in $x$.
We denote by $F_i$ a $p_i \circ \rho$-fibre and by $\bar S = \eta(F_i)$ its image\footnote{Note that the involution $i_{\blS}$ maps $F_1$ onto $F_2$, so $\barS$ is well-defined.} in $X$.
We will denote by $\bar S_x$ the image of the fibre $\fibre{p_i \circ \rho}{x} \subset \blS$ in $X$.

The tangent sequence for $\rho$
$$
0 \rightarrow \rho^* \Omega_{S \times S} \rightarrow \Omega_{\blS} \rightarrow \sO_E(2E) \rightarrow 0
$$
immediately yields
\begin{equation} \label{chernblowup}
\begin{array}{ll}
c_1(\Omega_{\blS}) = E, & 
c_3(\Omega_{\blS}) = E^3+ 24 (F_1+F_2) \cdot E,
\\
c_2(\Omega_{\blS}) = 24 (F_1+F_2) - E^2  &
c_4(\Omega_{\blS}) = - E^4 - 24  (F_1+F_2) \cdot E^2 + 576. 
\end{array}
\end{equation}
From tangent sequence for $\eta$
$$
0 \rightarrow \eta^* \Omega_X \rightarrow \Omega_{\blS} \rightarrow \sO_E(-E) \rightarrow 0
$$
one deduces 
\begin{equation} \label{chernX}
\begin{array}{ll}
c_1(\eta^* \Omega_X) = 0, &
c_3(\eta^* \Omega_X) = 0,
\\
c_2(\eta^* \Omega_X) = 24 (F_1+F_2) - 3 E^2,  & 
c_4(\eta^* \Omega_X) = 648.
\end{array}
\end{equation}
We can then deduce the Segre and Chern classes of $X$ :
\begin{equation} \label{segreX}
\begin{array}{ll}
s_1(X) = 0 = c_1(X), 
&  
s_3(X)=0 = c_3(X)
\\
s_2(X) = -24 \bar S + 3 \delta^2 = -c_2(X), 
&
s_2(X)^2=828=c_2(X)^2
\\
s_4(X) = 504, 
&
c_4(X) = 324.
\end{array}
\end{equation}
More precisely these formulas follow from \eqref{chernX}, the projection formula and the following lemmas.

\begin{lemma} \label{intersectionsS}
In the setup of subsection \ref{subsectionsetup}, one has
$$
\bar S \cdot \delta = l
$$
where $l$ is the class of a fibre of $\varepsilon|_{E_X} : E_X \rightarrow S$. Moreover one has
$$
\bar S \cdot \delta \cdot \alpha_X = 0, \qquad \bar S \cdot \delta^2=-1, \qquad \bar S^2 = 1, \qquad \bar S \cdot \alpha_X^2=\alpha_S^2.
$$
\end{lemma}

\begin{proof}
The first statement is equivalent to $\bar S \cdot E_X = 2l$.
Since $\bar S = \eta_* F_1$ and $\eta^* E_X = 2 E$ we know by the projection formula that
$$
\bar S \cdot E_X = \eta_* F_1 \cdot E_X = F_1 \cdot \eta^* E_X = 2 F_1 \cdot E. 
$$
Now recall that $F_i$ is the blow-up of $p \times S$ in the point $(p,p)$. Thus the intersection $F_1 \cdot E$ is the 
exceptional divisor of the blowup $F_i \rightarrow p \times S$. This exceptional $\PP^1$ 
maps isomorphically onto a fibre of $\varepsilon|_{E_X}$. This shows the first statement.

The equalities $\bar S \cdot \delta \cdot \alpha_X = 0, \ \bar S \cdot \delta^2=-1$ now follow from
\eqref{restrictalphaX} and \eqref{restrictdelta}.
Since $\eta^* \bar S = F_1 + F_2$ the projection formula implies
$$
\bar S^2 = \frac{1}{2} (\eta^* \bar S)^2 = \frac{1}{2} (F_1 + F_2)^2 = F_1 \cdot F_2 = 1,
$$
where the last equality is due to the fact that the strict transform of $p \times S$ and $S \times q$ intersect exactly
in $(p,q)$ if $p \neq q$.

Finally the equality $\bar S \cdot \alpha_X^2=\alpha_S^2$ follows from the construction of $\alpha_X$ \cite[Sect.6, Prop.6]{Bea83} and observing that if $F_{1,x}$ is the fibre of $p_1 \circ \rho$ over $x \in S$, then
$\alpha_X|_{\mu(F_{1,x})} = \rho_x^* \alpha_S$ where $\rho_x : F_1 \times S$ is the blow-up in $x$. 
\end{proof}

\begin{lemma} \label{intersectionsX}
In the setup of subsection \ref{subsectionsetup}, one has
$$
\begin{array}{lllll}
\alpha_X^4 = 3 (\alpha_S^2)^2,
&
\alpha_X^3 \cdot \delta = 0,
& 
\alpha_X^2 \cdot \delta^2 = - 2 \alpha_S^2,
&
\alpha_X \cdot \delta^3 = 0,
& 
\delta^4 = 12 
\end{array}
$$
\end{lemma}

\begin{proof}
A standard intersection computation based on \eqref{qsurface},  \eqref{restrictalphaX}, 
\eqref{restrictdelta} and $q(\delta)=-2$.
\end{proof}

\subsection{Positive threshold}

Using the preceding section we can easily compute the positive threshold:

\begin{corollary}\label{corollaryK3square}
Let $X$ be a four-dimensional Hyperk\"ahler manifold of deformation type $K3^{[2]}$. 
Let $\omega_X$ be a nef and big $(1,1)$-class on $X$ such that
$$
q(\omega_X)\geq 3+\sqrt{\frac{21}{5}}\sim 5.0493.
$$
Then $\zeta+\pi^*\omega_X$ is pseudoeffective.
This bound is optimal for a very general deformation of $X$.
\end{corollary}

\begin{proof}
By Proposition \ref{propositionsegre} we only have to compute the
largest real root of $p_X(t)$.  
By Formula \eqref{formulatop} and Example \ref{exampletop}
we have to compute the largest solution of 
$$
d_0 + 21 d_2 t + 35 d_4 t^2 = 0,
$$
where the constants $d_{2i}$ are defined by \eqref{intersectionsegre}.
By \eqref{qsurface} and \eqref{qctwo} we have
$$
c_2(X)\alpha^2=30 q(\alpha), \quad \alpha^4=3 q(\alpha)^2
$$
for any element $\alpha\in H^{1,1}(X, \R)$.
By \eqref{segreX} we have $c_4(X)=324$, $c_2^2=828$. Thus we 
obtain the quadratic equation 
$$
504 - 630 t + 105 t^2 = 0.
$$
Its largest solution is 
$$
C = \frac{630 + 42 \sqrt{105}}{210} = 3 + \sqrt{\frac{21}{5}}.
$$
\end{proof}

\begin{remark}
Let $X$ be a four-dimensional Hyperk\"ahler manifold, not necessarily deformation equivalent to a Hilbert square.
In this case the coefficients $d_i$ are not known. However, if a very general deformation of $X$ does not contain any subvarieties, we can use Example \ref{exampletop} to show that
for a K\"ahler class $\omega_X$ the 
positivity threshold is 
$$
\gamma_p(\omega_X)=\sqrt{
\frac{21 \omega_X^2 c_2+\sqrt{(21\omega_X^2c_2)^2- 140(\omega_X^4)(c_2^2-c_4) }}{70\omega_X^4}}.
$$
\end{remark}

\subsection{A subvariety of $\PP(\Omega_X)$} \label{subsectionZ}

Denote by $\holom{p_i}{S \times S}{S}$ the projection on the $i$-th factor. 
Then $p_i \circ \rho : \blS \rightarrow S$ is a submersion, 
the fibre over a point $x \in S$ being isomorphic to the blow-up of $S$ in $x$.
Thus we obtain rank two foliations
$$ 
\ker T_{p_i \circ \rho} =: \sF_i \subset T_{\blS}
$$
In view of the description of the $\sF_i$-leaves it is clear that the natural map $\sF_1 \oplus \sF_2 \rightarrow T_{\blS}$ has rank $4$ in the complement of the exceptional divisor $E$, but 
$$
\sF_1|_E \cap T_E = T_{E/S} = \sF_2|_E \cap T_E.
$$

\begin{lemma} \label{lemmainjection}
The composition of the inclusion $\sF_i \subset T_{\blS}$ with the tangent map $T_{\blS} \rightarrow \eta^* T_X$
is injective in every point. Thus $\sF_i \hookrightarrow \eta^* T_X$ is a rank 2 subbundle.
\end{lemma}

\begin{proof}
Since $T_\eta$ is an isomorphism in the complement of $E$, it is sufficient to study the restriction to $E$.
Note also that 
$$
T_{\blS}|_E \rightarrow (\eta^* T_X)|_E
$$
has rank three in every point, since $\eta|_E$ induces an isomorphism $E \rightarrow E_X$.
Arguing by contradiction we assume that there exists a point $x \in E$ such that the map 
$$
\sF_{i,x} \rightarrow T_{\blS,x} \rightarrow (\eta^* T_X)_x
$$
has rank at most one for some $i \in \{1, 2 \}$. Since $\eta \circ i_{\blS} = \eta$ this implies that 
$$
\sF_{3-i,x} \rightarrow T_{\blS,x} \rightarrow (\eta^* T_X)_x
$$
also has rank at most one. Yet $\ker T_{\eta,x}$ has dimension one, so we obtain 
$$
\ker T_{\eta,x} \cap \sF_{1,x} = \ker T_{\eta,x} = \ker T_{\eta,x} \cap \sF_{2,x}.
$$
In particular we have  
$$
\ker T_{\eta,x} = \sF_{1,x}  \cap \sF_{2,x}= T_{E/S,x}.
$$
Yet $\eta$ induces an isomorphism $E \rightarrow E_X$, so $T_{E/S,x} \subset T_{E,x}$ is not in the kernel.
\end{proof}

By Lemma \ref{lemmainjection} we have an injection $\sF_i \hookrightarrow \eta^* T_X$. The corresponding quotient
$\eta^* T_X \rightarrow Q_i$ defines a subvariety $\PP(Q_i)$ of $\holom{\pi_\eta}{\PP(\eta^* T_X)}{\blS}$ 
that is a $\PP^1$-bundle over $\blS$.
Since $\eta \circ i_{\blS} = \eta$ the involution $i_{\blS}^*$ acts on $\PP(\eta^* T_X)$ and maps
$\PP(Q_1)$ to $\PP(Q_2)$. Thus if we denote by $Z \subset \PP(T_X)$ the image of $Q_i$ under
the two-to-one cover $\tilde \eta : \PP(\eta^* T_X) \rightarrow \PP(T_X)$, we have
$$
\tilde \eta^* [Z] = [Q_1] + [Q_2].
$$ 
\begin{proposition}
In the situation of Subsection \ref{subsectionsetup}, denote by $Z \subset \PP(T_X) \simeq \PP(\Omega_X)$ the subvariety constructed above. Then we have
\begin{equation} \label{classZ}
[Z] = 2 \zeta^2 + 2 \pi^* \delta \cdot \zeta + \pi^* (24 \bar S-6 \delta^2).
\end{equation}
\end{proposition}

\begin{proof}
Consider the exact sequence
$$
0 \rightarrow \sF_i \rightarrow T_{\blS} \rightarrow (p_i \circ \rho)^* T_S \rightarrow 0.
$$
The Chern classes of $(p_i \circ \rho)^* T_S$ and $T_{\blS}$ are known, cf. \eqref{chernblowup}.  
An elementary computation then yields
\begin{equation} \label{chernFi}
c_1(\sF_i)=-E, \quad c_2(\sF_i) = 24 F_{3-i} - 3 E^2.
\end{equation}
Denote by $\zeta_\eta$ the tautological bundle on $\PP(\eta^* T_X)$.  
Since $Q_i = \eta^* T_X/\sF_i$ we have 
$$
[Q_i] = \zeta_\eta^2 - \zeta_\eta \cdot \pi_\eta^* c_1(\sF_i) + \pi_\eta^* c_2(\sF_i)
= 
 \zeta_\eta^2 + \zeta_\eta \cdot \pi_\eta^* E + \pi_\eta^* (24 F_{3-i} - 3 E^2).
$$
Since $\tilde \eta^* [Z] = [Q_1] + [Q_2]$ and
$$
\eta^* \bar S = F_1+F_2, \quad \eta^* \delta = E,  \quad \tilde \eta^* \zeta = \zeta_\eta
$$
the claim follows.
\end{proof}

\begin{remark}
The geometry of $Z$ can be understood as follows: on $\blS$ we have two distinct families
of surfaces $(\fibre{(p_i \circ \rho)}{x})_{x \in S}$. The images in $X$ of these two families
coincide
and form a web of surface $(\bar S_x)_{x \in S}$.  For a point $x \in X$ that is not in $E_X$
there are exactly two members of the web passing through $x$ and they intersect transversally.
The projectivisation of their normal bundle defines a projective line in $\PP(\Omega_{X,x})$.
Since the intersection is transversal, the general fibre of $Z \rightarrow X$ is thus a pair of disjoint
lines.

For a point $x \in E \subset X$, the involution $i_{\blS}^*$ acts on $\PP((\eta^* T_X)_x)$ and identifies 
$\PP(Q_{1,x})$ with $\PP(Q_{2,x})$.
Thus the fibre of $Z \rightarrow X$ over a point in $x \in E_X \simeq E$ is a double line.
Hence $Z \cap \pi^* E$ is non-reduced with multiplicity two. 
In fact since $(\eta^* T_X)|_E \simeq T_X|_{E_X}$ we can identify 
$(Z \cap \pi^* E)_{red}$ to the quotient defined by the inclusion
$\sF_i|_E \rightarrow (\eta^* T_X)|_E$.
\end{remark}

\subsection{The intersection computation} \label{subsectionnumbers}

We will now compute some intersection numbers on $\PP(\Omega_X)$. 

\begin{lemma} \label{lemmanumbersX}
In the situation of Subsection \ref{subsectionsetup}, let $\alpha_S$ be a $(1,1)$-class on $S$ 
and $\alpha_X=(i \otimes \id_\C)(\alpha_S) \in H^{1,1}(X, \R)$. Then one has
$$
\zeta^7 = 504,
$$
$$
\zeta^5 \cdot \pi^* \delta^2 = 60, \quad
\zeta^5 \cdot \pi^* (\delta \cdot \alpha_X) = 0, \quad
\zeta^5 \cdot \pi^* \alpha_X^2 = - 30 \alpha_S^2,
$$
$$
\zeta^3 \cdot \pi^* \delta^4 = 12, \quad 
\zeta^3 \cdot \pi^* (\delta^3 \cdot \alpha_X) = 0 , \quad 
\zeta^3 \cdot \pi^* (\delta^2 \cdot \alpha_X^2)  = - 2 \alpha_S^2, \quad 
$$
$$
\zeta^3 \cdot \pi^* (\delta \cdot \alpha_X^3) = 0, \quad 
\zeta^3 \cdot \pi^* \alpha_X^4 = 3 (\alpha_S^2)^2.
$$
\end{lemma}

\begin{proof}
Observe first that $\zeta^7= s_4(X)$, so the first statement is included in \eqref{segreX}.
Also note that by \eqref{segreX} one has
$$
\pi_* \zeta^5 = s_2(X) = - 24 \bar S + 3 \delta^2,
$$
so the second statement follows from Lemma \ref{intersectionsS} and Lemma \ref{intersectionsX}.
The intersections with $\zeta^3$ are simply a restatement of Lemma \ref{intersectionsX}.
\end{proof}

In order to compute the intersection numbers with $\pi^* \bar S$, note that by Lemma \ref{intersectionsS}
one has 
$$
c_1(\Omega_X|_{\bar S})= 0 , \quad 
s_2(\Omega_X|_{\bar S}) = s_2(\Omega_X) \cdot \bar S = (-24 \bar S + 3 \delta^2) \cdot \bar S = - 27.
$$
Thus we have $\zeta^5 \cdot \pi^* \bar S= -27$ and
\begin{equation} \label{up7}
\zeta^3 \cdot \pi^* \bar S \cdot \delta^2 = -1, \quad
\zeta^3 \cdot \pi^* \bar S \cdot \alpha_X \cdot \delta = 0, \quad
\zeta^3 \cdot \pi^* \bar S \cdot \alpha_X^2 = \alpha_S^2.
\end{equation}

The intersections with $\zeta^4$ and $\zeta^6$ are all equal to zero: the
Segre classes $s_1(X)$ and $s_3(X)$ vanish, so the statement follows from the projection formula.

Let now $S$ be a very general K3 surface such that $\mbox{Pic}(S)=0$, in particular $S$ does not contain any curves.
The subvarieties of the product $S \times S$ are exactly $S \times x, x \times S$ and the diagonal $\Delta$: the case
of curves and divisors is easily excluded. For a surface $Z \subset S \times S$ we first observe
that the projection on $S$ is \'etale, since $S$ does not contain any curve. Since $S$ is simply connected, we obtain
that $Z$ is the graph of an automorphism of $S$. Yet a very general K3 surface has no non-trivial automorphisms \cite[Cor.1.6]{Ogu08}.

\begin{lemma} \label{lemmakaehlerX}
In the situation of Subsection \ref{subsectionsetup}, 
let $S$ be a very general K3 surface such that $\mbox{Pic}(S)=0$. 
\begin{itemize}
\item The subvarieties of $X$ are exactly $(\bar S_x)_{x \in S}$, the exceptional divisor $E_X$ and the fibres of $E_X \simeq \PP(\Omega_X) \rightarrow S$.
\item Let $\alpha_S$ be a K\"ahler class on $S$. Then $\alpha_X-\delta$ is a K\"ahler class if and only if
$\alpha_S^2>2$.
\end{itemize}
\end{lemma}

\begin{proof}
Since $\eta$ is finite, any subvariety of $X$ corresponds to a subvariety of $\blS$. By the discussion above 
and Corollary \ref{corollarynosubvariety}
we  know the subvarieties of $S \times S$ and $\PP(\Omega_S)$, so the first statement follows.

We know that $t \alpha_X-\delta$ is K\"ahler for $t \gg 0$, so by the Demailly-P\v aun theorem it is enough to check when 
$\alpha_X-\delta$ is in the positive cone. By Lemma \ref{intersectionsS} and Lemma \ref{intersectionsX} we have
$$
(\alpha_X-\delta)^4 = 3 ((\alpha_S^2)^2-4\alpha_S^2+4), \quad
(\alpha_X-\delta)^3 \cdot E = 12 (\alpha_S^2-2)
$$
$$
(\alpha_X-\delta)^2 \cdot \barS = \alpha_S^2-1, \quad
(\alpha_X-\delta) \cdot l = 1,
$$
which are all positive for $\alpha_S^2>2$.
\end{proof}

\begin{proposition} \label{propositionthenumber}
In the situation of Subsection \ref{subsectionsetup}, let
$\alpha_S$ be a K\"ahler class on $S$ such that $\omega := \alpha_X-\delta$ is a K\"ahler class.
Let $Z \subset \PP(T_X) \simeq \PP(\Omega_X) $ be the subvariety constructed in Subsection \ref{subsectionZ}. Then we have
$$
(\zeta + \pi^* \omega)^5 \cdot [Z] =
15 \left(
(\alpha_S^2)^2 - 8 \alpha_S^2 - 56
\right).
$$
In particular we have
$$
(\zeta + \pi^* \omega)^5 \cdot [Z]  \geq 0
$$
if and only if $\alpha_S^2 \geq \frac{8+\sqrt{288}}{2} \approx 9,6569$.
\end{proposition}

\begin{proof}
The class $[Z]$ is given by \eqref{classZ} and all the intersection numbers are determined in Subsection \ref{subsectionnumbers}. The statement follows from an elementary, but somewhat lengthy computation.
\end{proof}

We can summarise our computations on $X=S^{[2]}$ as follows:
since $\alpha_S^2=q(\alpha_X)$ we know that 
for a very general K3 surface, the class $\alpha_X-\delta$ is K\"ahler if $q(\alpha_X)>2$ (Lemma \ref{lemmakaehlerX}).
The class $\zeta+\pi^*(\alpha_X-\delta)$ is pseudoeffective if $q(\alpha_X)\geq 5+\sqrt{\frac{21}{5}}$ (Corollary \ref{corollaryK3square}).
If $q(\alpha_X) < \frac{8+\sqrt{288}}{2}$, the class $\zeta+\pi^*(\alpha_X-\delta)$ is not nef (Proposition \ref{propositionthenumber}).
In particular we see that for the Hilbert square of a K3 surface polarised by an ample line bundle $L$ of degree eight, the integral class $\zeta+\pi^*(c_1(L)_X-\delta)$ is big but not nef.

\subsection{Remark on subvarieties of $X$}

By \cite[Thm.1.1]{Ver98} a very general deformation of the Hilbert scheme $S^{[n]}$
does not contain any proper subvarieties. Verbitsky's proof is rather involved,
but for the case $n=2$ general arguments are sufficient:
a very general deformation satisfies satisfies $\operatorname{Pic}(X)=0$, so there are no 
divisors and by duality there are no curves on $X$.
The vector space
$$
H^4(X,\Q) \cap H^{2,2}(X)
$$
is one dimensional by \cite[Table B.1]{Zha15} and thus generated by the non-zero class $c_2(X)$.
If $X$ contains a surface $S$, we obtain that $c_2(X)$ is represented by an effective $\Q$-cycle for $X$ very general.
By properness of the relative Barlet space \cite[Theorem 4.3]{Fujiki78} this implies that $c_2(X)$ is effectively
represented for every member in the deformation family. Yet this contradicts \cite[Proposition 2]{Ott15}.

\section{Hilbert cube of a K3 surface}

Now we compute explicitly the positivity threshold for $n=3$.

\begin{corollary} \label{corollaryK3cube}
Let $X$ be a six-dimensional Hyperk\"ahler manifold of deformation type $K3^{[3]}$. 
Let $\omega_X$ be a nef and big $(1,1)$-class on $X$ such that
$$
q(\omega_X)\geq \frac{2}{21}(18+\sqrt[3]{6(1875-7\sqrt{4233})}+\sqrt[3]{6(1875+7\sqrt{4233})}) \approx 5.9538
$$
Then $\zeta+\pi^*\omega_X$ is pseudoeffective.
This bound is optimal for a very general deformation of $X$.
\end{corollary}

The proof is based on the following proposition, communicated to us by Samuel Boissi\`ere :

\begin{proposition} \label{propositionK3cube} (S. Boissi\`ere)
Let $X$ be a six-dimensional Hyperk\"ahler manifold of deformation type $K3^{[3]}$. Then for any $(1,1)$-class
$\alpha$ on $X$ one has
$$
\alpha^6=15q(\alpha), \qquad c_2 \alpha^4=108q(\alpha)^2
$$
$$
c_2^2 \alpha^2=1848 q(\alpha), \qquad c_4 \alpha^2=2424 q(\alpha).
$$
\end{proposition}

\begin{proof}
By \cite[Remark 4.12]{F87} or \cite[Theorem 5.12]{Huy97} we know that for any element $\gamma\in H^{i,i}(X, \R)$ that deforms to a very general deformation of $X$ as element of type $(i,i)$, its intersection with a class in $H^2(X, \R)$ satisfies $\gamma\cdot \alpha =C(\gamma)q(\alpha)^{n-i}$ for any $\alpha \in H^2(X,\Z)$. We need to compute these constants for varieties that are deformation equivalent to $K3^{[3]}$ and $\gamma$ in the subalgebra generated by the Chern classes. By abuse of notations we will denote by $c_i$ the Chern classes of $X$. 
The constants $C(\gamma)$ are invariant by deformations, so we can assume that $X$ is isomorphic to $S^{[3]}$ for a projective K3 surface $S$. As we mention before in the case of $S^{[2]}$, there is an isometric inclusion $i:H^2(S,\Z)\hookrightarrow H^2(X, \Z)$. Geometrically this inclusion is realized sending a line bundle $L$ on $S$ to the line bundle $L_3:=det L^{[3]}$. By Riemann--Roch formula and by \cite[Lemma 5.1]{EGL01} we have
$$\int_X e^{c_1(L_3)} \operatorname{Todd}(X) =\chi_X (L_3)=\binom{\chi_S(L)+2}{2}.$$
From now on we by abuse of notation we will confuse line bundles with their first Chern class. We recall that the Todd class for six dimensional Hyperk\"ahler manifolds is
\begin{equation}\label{todd}
\operatorname{Td}(X)=1+\frac{1}{12}c_2+\frac{1}{240}c_2^2-\frac{1}{720}c_4+\frac{1}{6048}c_2^3-\frac{1}{6720}c_2c_4+\frac{1}{30240}c_6
\end{equation}
and $\chi_S(L)=L^2+2=q(L_3)+2$.
Putting the Todd class and the characteristics in the equation above we get 
$$\frac{1}{720}L_3^6+\frac{1}{288}c_2 L_3^4+(\frac{1}{480}c_2^2-\frac{1}{1440}c_4)L_3^2+\frac{1}{6048}c_2^3-\frac{1}{6720}c_2c_4+\frac{1}{30240}c_2c_4=$$
$$=\frac{1}{6}\chi_S(L)(\chi_S(L)+1)(\chi_S(L)+2)=\frac{1}{48}q(L_3)^3+\frac{3}{8} q(L_3)^2+\frac{13}{6}q(L_3)+4$$
that by homogeneity tells us that 
$$L_3^6=15q(L_3)$$ 
$$c_2 L_3^4=108q(L_3)^2.$$
The quadratic term is not sufficient to gives us the other constants but tells only that 
\begin{equation}\label{chernintersectionone}
3c_2^2L_3^2-c_4L_3^2=3120q(L_3).
\end{equation}
We are going to use a consequence of a formula due to Nieper that can be found in \cite[Theorem 4.2]{Huy03b}:
\begin{equation}\label{nieper}
\int_X\sqrt{\operatorname{Td}(X)}e^x=(1+\lambda(x))^3\int_X\sqrt{\operatorname{Td}(X)}
\end{equation}
for a quadratic form $\lambda:H^2(X,\C)\rightarrow\C$ and any $x\in H^2(X,\C)$.
One can deduce directly by \eqref{todd} that 
$$\sqrt{\operatorname{Td}(X)}=1+\frac{1}{24}c_2+\frac{7}{5650}c_2^2-\frac{1}{1440}c_4+\frac{31}{967680}c_2^3-\frac{11}{241920}c_2c_4+\frac{1}{60480}c_6.$$
By the terms of degree 4 and 6 of \eqref{nieper} we deduce that $\lambda(x)=\frac{1}{3}q(x)$. This fact with the degree two component of \eqref{nieper} gives 
\begin{equation}\label{chernintersectiontwo}
	\frac{7}{4}c_2^2x^2-c_4x^2=810 q(x).
\end{equation}
Finally the solution of the system given by \eqref{chernintersectionone} and \eqref{chernintersectiontwo} is 
$$
c_2^2 L_3^2=1848 q(L_3), \qquad 
c_4 L_3^2=2424 q(L_3).
$$
\end{proof}

\begin{proof}[Proof of Corollary \ref{corollaryK3cube}]
By Proposition \ref{propositionsegre} we only have to compute the
largest real root of $p_X(t)$.  
By Formula \eqref{formulatop}
we have to compute the largest solution of 
$$
\binom{11}{6} d_6 t^3 + \binom{11}{4} d_4 t^2 + \binom{11}{2} d_2 t + d_0 = 0,
$$
where the constants $d_{2i}$ are defined by \eqref{intersectionsegre}.
Using Proposition \ref{propositionK3cube} we can compute the constant $d_{2i}$ in our setting,
one obtains the cubic equation 
$$
6930 t^3-35640 t^2-31680 t-10560. = 0.
$$
This polynomial has only one real solution, the one from the statement.
The last statement is the second part of Theorem \ref{theoremmain}. 
\end{proof}

\begin{appendix} 
\section{Limits in family of pseudoeffective classes. \\ By Simone Diverio\footnote{\textsc{Simone Diverio, Dipartimento di Matematica \lq\lq Guido Castelnuovo\rq\rq{}, SAPIENZA Universit\`a di Roma, I-00185 Roma,} \emph{E-mail address}: \texttt{diverio@mat.uniroma1.it}. \\
Partially supported by the ANR Programme D\'efi de tous les savoirs (DS10) 2015, \lq\lq GRACK\rq\rq, Project ID: ANR-15-CE40-0003ANR, and by the ANR Programme D\'efi de tous les savoirs (DS10) 2016, \lq\lq FOLIAGE\rq\rq, Project ID: ANR-16-CE40-0008.}}
\label{appendixdiverio}

Let $\pi\colon\mathfrak X\to\Delta$ be a proper holomorphic submersion onto the complex unit disc of relative complex dimension $n$, and call $X_t=\pi^{-1}(t)$ the compact complex manifold over the point $t\in\Delta$. 

Suppose also that $\pi$ is a weakly K\"ahler fibration, \textsl{i.e.} there exists a real $2$-form $\omega$ on $\mathfrak X$ such that its restriction $\omega_t=\omega|_{X_t}$ is a K\"ahler form on $X_t$, for each $t\in\Delta$.

By Ehresmann's fibration theorem, $\pi$ it is a locally trivial fibration in the smooth category. Thus, after possibly shrinking $\Delta$, we may suppose that we are given a smooth compact real manifold $F$ of real dimension $2n$ and a smooth diffeomorphism $\theta\colon\mathfrak X\to F\times\Delta$ such that the following diagram commutes:
$$
\xymatrix{
\mathfrak X\ar[rr]^{\theta} \ar[dr]_{\pi}& & F\times\Delta \ar[dl]^{\textrm{pr}_2} \\
& \Delta.
}
$$
Next, call $\theta_t:=\theta|_{X_t}\colon X_t\underset{\simeq_{C^\infty}}\longrightarrow F$. For any $t\in\Delta$, given a real $(1,1)$-cohomology class $\alpha_t\in H^{1,1}(X_t,\mathbb R)$, we can then think of it as an element $\beta_t$ of $H^2(F,\mathbb R)$, by pulling-back \textsl{via} $\theta_t^{-1}$, that is $\beta_t:=\bigl(\theta_t^{-1}\bigr)^*\alpha_t$.

Now, suppose that we are given a class $\alpha_0\in H^{1,1}(X_0,\mathbb R)$ with the following property: there is a sequence of points $\{t_k\}\subset\Delta$ converging to $0$, for each $k$ it is given a $(1,1)$-class $\alpha_{t_k}\in H^{1,1}(X_{t_k},\mathbb R)$ which is pseudoeffective and the corresponding classes $\beta_{t_k}$ converge to $\beta_0$ in the finite dimensional vector space $H^2(F,\mathbb R)$. Then we have the following statement.

\begin{theorem} \label{limitpseff}
The class $\alpha_0$ is also pseudoeffective.
\end{theorem}

We claim in no way any originality for this theorem, since this is certainly a well-known statement for the experts, and moreover widely used. But we were unable to find a clean proof in the available literature. We take thus the opportunity here to give a complete proof. We follow the notations of Demailly's book \cite{Dem12}.  

\begin{proof}
To start with, we select for each $k$ a closed, positive $(1,1)$-current $T_{k}\in\mathfrak D'^+_{n-1,n-1}(X_{t_k})$ representing the cohomology class $\alpha_{t_k}$. Each of these, being a positive current, is indeed a real current of order zero.

Now, set $\Theta_{k}:=\bigl(\theta_{t_k}\bigr)_*T_k$. This is a closed, real $2$-current of order zero on the compact real smooth manifold $F$. 

The first step is to produce a weak limit $\Theta$ of the sequence $\Theta_k$ on $F$. In order to to this, by the standard Banach--Alaoglu theorem, it suffices to show that for every fixed test form $g\in\mathfrak D^{2n-2}(F)$ we have that the sequence $\langle \Theta_k,g\rangle$ is bounded. By definition, we have
$$
\langle \Theta_k,g\rangle= \bigl\langle \bigl(\theta_{t_k}\bigr)_*T_k,g\bigr\rangle=\bigl\langle T_k,\bigl(\theta_{t_k}\bigr)^*g\bigr\rangle,
$$
and of course $\bigl\langle T_k,\bigl(\theta_{t_k}\bigr)^*g\bigr\rangle=\langle T_k,f_k\rangle$, where $f_k$ is the $(n-1,n-1)$ component of $\bigl(\theta_{t_k}\bigr)^*g$ on the complex manifold $X_{t_k}$. The $(n-1,n-1)$-forms $f_k$ are real, since $\bigl(\theta_{t_k}\bigr)^*g$ is so.

\begin{lemma}
Let $(X,\omega)$ be a compact K\"ahler manifold, $T$ be a closed positive current of $X$, and $f$ be a real smooth $(n-1,n-1)$-form. Then, there exists a constant $C>0$ depending continuously on $f$ and $\omega$ such that we have
$$
|\langle T, f\rangle|\le C\, [T]\cdot[\omega]^{n-1},
$$
where the right hand side is intended to be the intersection product in cohomology.
\end{lemma}

\begin{proof}
Since $f$ is real, we are enabled to define the following (possibly indefinite) hermitian form on $T_X^*$:
$$
(\xi,\eta)_f\mapsto \frac{f\wedge i\xi\wedge\bar\eta}{\omega^n}.
$$
We also have the positive definite hermitian form given by
$$
(\xi,\eta)_\omega\mapsto \frac{\omega^{n-1}\wedge i\xi\wedge\bar\eta}{\omega^n}.
$$
It is positive because $(\xi,\xi)_\omega=\frac 1n\operatorname{tr}_\omega(i\,\xi\wedge\bar\xi)$. 
By compactness of the bundle of $(\cdot,\cdot)_\omega$-unitary $(1,0)$-forms on $X$, we can define
$$
C':=-\min_{(\xi,\xi)_\omega=1}\{(\xi,\xi)_f\},
$$
and we have that $(\xi,\eta)'\mapsto (\xi,\eta)_f+C'\,(\xi,\eta)_\omega$ is positive semidefinite. This constant $C'$ depends manifestly continuously on $f$ and $\omega$. We can do the same job with $-f$ in the place of $f$ thus obtaining another constant $C''$, still depending continuously on $f$ and $\omega$ such that
$$
(\xi,\eta)''\mapsto -(\xi,\eta)_f+C''\,(\xi,\eta)_\omega
$$
is positive semidefinite. Now set $C:=\max\{C',C''\}\ge 0$, which again depends continuously on $f$ and $\omega$. This means exactly that both $f+C\,\omega^{n-1}$ and $-f+C\,\omega^{n-1}$ are positive $(n-1,n-1)$-forms.

But then, begin $T$ positive on positive forms,
$$
\begin{aligned}
\langle T, f\rangle &= \langle T, f+C\,\omega^{n-1}-C\,\omega^{n-1}\rangle \\
&\ge -C\,\langle T, \omega^{n-1}\rangle=-C\, [T]\cdot[\omega]^{n-1},
\end{aligned}
$$
and
$$
\begin{aligned}
\langle T, f\rangle &= \langle T, f-C\,\omega^{n-1}+C\,\omega^{n-1}\rangle \\
&= -\langle T,-f+C\,\omega^{n-1}\rangle+\langle T, C\,\omega^{n-1}\rangle\\
& \le C\,\langle T, \omega^{n-1}\rangle=C\, [T]\cdot[\omega]^{n-1}.
\end{aligned}
$$
\end{proof}

Now, we apply the above lemma with $(X,\omega)=(X_{t_k},\omega_{t_k})$, $T=T_k$ and $f=f_k$. We therefore obtain positive constants $C_k$ such that
$$
|\langle T_k,f_k\rangle|\le C_k\,[T_k]\cdot[\omega_{t_k}]^{n-1}.
$$
The right hand side is equal to $C_k\,\beta_{t_k}\cdot\Omega_{t_k}^{n-1}$, where $\Omega_t\in H^2(F,\mathbb R)$ is the cohomology class of $\bigl(\theta_t^{-1}\bigr)^*\omega_t$. It converges to the quantity $C_0\,\alpha_0\cdot\Omega_0$, where $C_0$ is the constant obtained if one applies the above lemma with $(X,\omega)=(X_{0},\omega_{0})$, and $f$ the $(n-1,n-1)$ component of $\bigl(\theta_{0}\bigr)^*g$. Thus, the left hand side is uniformly bounded independently of $k$.

We finally come up with a real $2$-current $\Theta$ on $F$ which is a weak limit of the $\Theta_k$'s. By continuity of the differential with respect to the weak topology we find also that $\Theta$ is closed and of course its cohomology class is $\beta_0$. Being $\Theta$ trivially with compact support since it lives on the compact manifold $F$, by \cite[Corollary on p. 43]{dRh84}, it is of finite order, say of order $p$.

\begin{remark}\label{rk:BS}
We can then look at the whole sequence $\{\Theta_k\}$ together with its weak limit $\Theta$ as a set of currents of order $p$. In particular, this is a set of continuous linear functionals on the Banach space $^p\mathfrak D^{2n-2}(F)$ which are pointwise bounded. By the Banach--Steinhaus theorem this set is uniformly bounded in operator norm, \textsl{i.e.} there exists a constant $A>0$ such that for each positive integer $k$ and each $g\in{}^p\mathfrak D^{2n-2}(F)$ we have
$$
|\Theta_k(g)|\le A\, ||g||_{^p\mathfrak D^{2n-2}(F)}.
$$
This remark will be crucial in what follows.
\end{remark}

Next, set $T:=\bigl(\theta_0^{-1}\bigr)_*\Theta$. It is a real current of degree 2 on $X_0$. We are left to show that $T$ is indeed a $(1,1)$-current which is moreover positive.

\begin{proposition}
The current $T$ is of pure bidegree $(1,1)$.
\end{proposition}

\begin{proof}
If not, there exists a $(n,n-2)$-form $h$ on $X_0$ such that $\langle T,h\rangle\ne 0$. Fix a finite open covering of $X_0$ by coordinate charts and a partition of unity $\{\varphi_j\}$ relative to this covering. Since
$$
0\ne\langle T,h\rangle=\bigl\langle T,\sum_j\varphi_jh\bigr\rangle=\sum_j\langle T,\varphi_jh\rangle,
$$
there exists a $j_0$ such that $\langle T,\varphi_{j_0}h\rangle\ne 0$. Thus, we may assume that $h$ is compactly supported in a coordinate chart $(U,z)$. Without loss of generality, we can also suppose that such a coordinate chart is adapted to the fibration $\pi$, \textsl{i.e.} $U=\mathcal U\cap X_0$, where $\mathcal U$ is a coordinate chart for $\mathfrak X$ with coordinates $(t,z)$ such that $\pi(t,z)=t$.

In this way, we can extend $h$ \lq\lq constantly\rq\rq{} on the nearby fibres of $\pi$: call this extension $\tilde h$ and write $\tilde h_t$ for $\tilde h|_{\mathcal U\cap X_t}$. If we set $u_t:=\bigl(\theta_t^{-1}\bigr)^*\tilde h_t$ we obtain a family of test form on $F$ such that, for $k$ sufficiently large, we have
$$
\langle T_k,\tilde h_{t_k}\rangle=\langle\Theta_k,u_{t_k}\rangle.
$$
By Remark \ref{rk:BS}, we have
$$
\begin{aligned}
|\langle T_k,\tilde h_{t_k}\rangle-\langle T,\tilde h\rangle| &= |\langle\Theta_k,u_{t_k}\rangle-\langle\Theta,u_{0}\rangle| \\
&\le |\langle\Theta_k,u_{t_k}-u_0\rangle|+|\langle\Theta_k,u_0\rangle-\langle\Theta,u_0\rangle|\\
&\le A\,\underbrace{||u_{t_k}-u_0||_{^p\mathfrak D^{2n-2}(F)}}_{\textrm{$\to 0$, by construction}}+\underbrace{|\langle\Theta_k,u_0\rangle-\langle\Theta,u_0\rangle|}_{\textrm{$\to 0$, by weak convergence}}.
\end{aligned}
$$
Being the $T_k$'s of bidegree $(1,1)$ and $\tilde h_{t_k}$ of bidegree $(n,n-2)$, we have that $\langle T_k,\tilde h_{t_k}\rangle\equiv 0$ and we deduce then that $\langle T,h\rangle=0$, contradiction. 
\end{proof}

\begin{proposition}
The current $T$ is positive.
\end{proposition}

\begin{proof}
The proof is almost identical to that of the above proposition. We want to show that for any positive $(n-1,n-1)$-form $h$ on $X_0$ we have that $\langle T,h\rangle\ge 0$. As before the question is local, so we can suppose that $h$ is compactly supported in $U$ as above. Now the \lq\lq constant\rq\rq{} extensions $\tilde h_t$ are again positive $(n-1,n-1)$-forms on $X_t$, so that $\langle T_k,\tilde h_{t_k}\rangle\ge 0$ and we still have convergence to $\langle T,h\rangle$. But then $\langle T,h\rangle\ge 0$.
\end{proof}

This concludes the proof of the theorem, since we have represented $\alpha_0$ by a closed positive $(1,1)$-current, \textsl{i.e.} $\alpha_0$ is a pseudoeffective class.
\end{proof}

\end{appendix}

\end{document}